\documentclass[12pt,a4paper]{amsart}
\usepackage{dsfont}
\usepackage{graphicx}
\usepackage{amssymb}
\usepackage[all]{xy}
\usepackage{amsmath}
\usepackage{tikz}
\usepackage[french,english]{babel}
\usepackage{amssymb,amsmath,amsfonts,amsthm,amscd}
\usepackage{indentfirst,graphicx}
\usepackage[font={scriptsize},captionskip=5pt]{subfig}
\usepackage{tikz}
\usepackage{tikz-cd}
\usetikzlibrary{matrix}

\newtheorem{theorem}{Theorem}[section]
\newtheorem{corollary}[theorem]{Corollary}
\newtheorem{proposition}[theorem]{Proposition}
\newtheorem{lemma}[theorem]{Lemma}

\newtheorem{definition}[theorem]{Definition}

\newtheorem{example}[theorem]{Example}
\newtheorem{remark}[theorem]{Remark}

\newcommand{\cB}{{\mathcal B}}

\newcommand{\cF}{{\mathcal F}}

\newcommand{\cO}{{\mathcal O}}

\newcommand{\cR}{{\mathcal R}}

\newcommand{\cU}{{\mathcal U}}

\newcommand{\cX}{{\mathcal X}}

\newcommand{\cM}{{\mathcal M}}

\newcommand{\cOXp}{{\cO_X^p}}
\newcommand{\bC}{{\mathbb{C}}}
\newcommand{\bP}{{\mathbb{P}}}

\newcommand{\Projan}{{\mbox{Projan}}}
\newcommand{\cRM}{\cR(\cM)}
\newcommand{\cRMD}{\cR(\cM_D)}
\newcommand{\tM}{\mathbf{2}\cM}
\newcommand{\cRtM}{\cR(\mathbf{2}\cM)}

\newcommand{\projmxd}{\mbox{Projan}(\cR(JM(X)_D))}

\begin{document}

\title{The projective analytic spectrum of the double of a module}
\author{Terence Gaffney and Thiago da Silva}

\maketitle

\begin{abstract}
	{\small In this work, we investigate the projectivized analytic spectrum of the double of a module, establishing some general properties, and we apply these results to  $\mbox{Projan}(\cR((JM(X))_D))$ over the origin in $C\times C$, where $C$ is an irreducible curve in a hypersurface $X$.}
\end{abstract}

\thispagestyle{empty}

\section*{Introduction}

The concept of Lipschitz saturation of an ideal is introduced in \cite{G2}, within the framework of bi-Lipschitz equisingularity. The investigation of bi-Lipschitz equisingularity began with the work of Zariski \cite{Z}, followed by contributions from Pham and Teissier \cite{PT}, and was subsequently advanced by Lipman \cite{L}, Mostowski \cite{M1,M2}, Parusinski \cite{PA1}, Birbrair \cite{B}, among others.

The notion of the double of an ideal was introduced in \cite{G1}, and later generalized to sheaves of modules in \cite{SG}. The purpose of the double is to capture infinitesimal conditions that may indicate whether a given family of complex analytic varieties is Lipschitz. In \cite{SG}, the authors proved that a family of strongly bi-Lipschitz analytic varieties (in the sense of Ruas and Fernandes \cite{FR2, FR}) necessarily satisfies the infinitesimal Lipschitz condition — a criterion involving the integral closure of the double of a Jacobian submodule. Further results concerning the double in the context of Lipschitz geometry can also be found in \cite{dasilva, SGP, SGP1, SG1}.

The double of an submodule $M\subseteq\cOXp$, denoted $M_D$ is the submodule of $\cO_{X\times X}^{2p}$ generated by $(h\circ\pi_1,h\circ\pi_2)$, $h\in M$, where $\pi_1,\pi_2:X\times X\rightarrow X$ are the projections. 

In \cite[Prop. 3.11]{SG}  stated that the stalk of the double of a sheaf of modules $\cM$ at $(x,x')$, $x\ne x'$, is the direct sum of the stalks of $\cM$ at $x$ and $x'$. Thus, the stalk of the double carries the same information as the stalks of $\cM$ do at $x$ and $x'$, as long $x\neq x'$. If $\cM$ is the jacobian module of a family of analytic varieties, the stalks at $x$ and $x'$ determine the tangent hyperplanes at these two points. Since to control the Lipschitz behavior of the tangent hyperplanes to $X$, it is natural to look for a sheaf on $X\times X$ whose stalks determine the tangent hyperplanes at each pair of distinct points, it is natural to consider the double of the jacobian module. This naturally leads us to investigate the projective analytic spectrum of the double of the double of the Jacobian module associated to $X$, where we will see that it provides the appropriate setting to understand how these pairs of tangent planes relate to one another.

In the Section 1 we recall some definitions and basic results about the double of a module. In Section 2, we establish some general results about the projectivized analytic spectrum of the double of $\cM$, reducing the complexity of computing its fibers at points off the diagonal to the fibers of the spectrum associated with other module induced by $\cM$. We ultimately show that the desired fiber is the join of the two fibers of the spectrum of $\cM$ at that pair of points, that is: $$\Projan(\cRtM)_{(x,x')}\cong \Projan(\cRM)_x\ast \Projan(\cRM)_{x'}.$$

Finally, in Section 3, we investigate more concretely the fibers of the Projan over the origin for curves defined on hypersurfaces. In the end, we give special attention to plane curves in Theorem \ref{T3}.

\section{Background on the double of a module}

We recall some basic results about the double of a module developed in \cite{SG}. Let $X\subseteq\bC^n$ be an analytic space. We denote $\cO_X$ the holomorphic sheaf of rings over $X$. 

Let $\cM$ be an $\cO_X $-submodule of $\cO_X^p$. Consider the projection maps $\pi_1,\pi_2: X \times X \rightarrow X$. We assume that $\cM$ is finitely generated by global sections.

\begin{definition}
	Let $h\in\cOXp$. The double of $h$ is defined as the element
	
	\noindent $h_D:=(h\circ\pi_1,h\circ\pi_2) \in\cO_{X\times X}^{2p}$.
	
	The double of $\cM$ is denoted by $\cM_D$, and is defined as the $\cO_{X\times X}$-submodule of $\cO_{X\times X}^{2p}$ generated by $\{h_D \mid  h\in \cM(X)\}$.
\end{definition}

Consider $z_1,\hdots ,z_n$ the coordinates on $\mathbb{C}^n$. The next lemma is a useful tool to deal with the double. In \cite{SG} we see that it is possible to obtain a set of generators for $\cM_D$ from a set of generators of $\cM$.

\begin{proposition}[\cite{SG}, Proposition 3.3]\label{P2}
	Suppose that $\cM$ is generated by $\{h_1,\ldots,h_r \}$. Then, the following sets are generators of $\cM_D$:
	\begin{enumerate}
		
		\item $\cB=\{(h_1)_D,\ldots,(h_r)_D\} \cup \{(0_{\cO_{X\times X}^{p}},(z_i\circ\pi_1-z_i\circ\pi_2)(h_j\circ\pi_2))$ $|$ $i\in\{1,\ldots,n\}$ and $j\in\{1,\ldots,r\}\}$.
		
		\item $\cB'=\{(h_1)_D,\ldots,(h_r)_D\} \cup \{((z_i\circ\pi_1-z_i\circ\pi_2)(h_j\circ\pi_1),0_{\cO_{X\times X}^{p}}))$ $|$ $i\in\{1,\ldots,n\}$ and $j\in\{1,\ldots,r\}\}$.
		
		\item $\cB''=\{(h_1)_D,\ldots,(h_r)_D\} \cup \{(z_ih_j)_D$ $|$ $i\in\{1,\ldots,n\}$ and $j\in\{1,\ldots,r\}\}$.
	\end{enumerate}
\end{proposition}

The next theorem computes the generic rank of the double of a module.

\begin{proposition}[\cite{SG}, Proposition 3.5]\label{T2.9}
	Let $(X,x)$ be an irreducible analytic complex germ of dimension $d\ge 1$, and $\cM\subseteq\cO_{X,x}^p$ a submodule of generic rank $k$. Then $\cM_D$ has generic rank $2k$ at $(x,x)$. 
\end{proposition}

Then next proposition states that the double of a sheaf of modules $\cM$ carries all the information at $(x,x')$ as the stalks of $\cM$ do at $x$ and $x'$, as long $x\neq x'$.

\begin{proposition}[\cite{SG}, Proposition 3.11]\label{P2.13}
	Let $\cM\subseteq\cO_X^p$ be a sheaf of submodules. Consider 
	$(x,x')\in X\times X$ with $x\neq x'$. Then:
	\begin{center}$\cM_D=(\cM_x\circ\pi_1)\oplus (\cM_{x'}\circ\pi_2)$\end{center} at $(x,x')$.
\end{proposition}

Proposition \ref{P2.13} provides additional motivation for the idea of the double: In order to control the Lipschitz behavior of pairs of tangent planes at two different points $x$ and $x'$ of a family $\cX$, it is helpful to have each module which determines the tangent hyperplanes at each point as part of the construction. Furthermore, this proposition shows that $JM(\cX)_D$ at $(x,x')$ contains both $JM(\cX)_x$ and $JM(\cX)_{x'}$. 

Now we recall the notion of the singular set of a sheaf of modules $\cM$. 

Let $\cM$ be a sheaf of $\cO_X$-submodules of $\cOXp$ generated by global sections $\{g_1,\hdots ,g_r\}$. The matrix $[\cM]$ is the $p\times r$ matrix whose columns are the generators of $\cM$. For each $x\in X$, $[\cM(x)]$ is the $p\times r$ matrix obtained by applying each entry of $[\cM]$ at $x$. We denote $\mbox{row}[\cM(x)]$ the rowspace of the matrix $[\cM(x)]$. Suppose the generic rank of $\cM$ is $k$. The singular set of $\cM$ is defined as the set $$\Sigma(\cM):=\{x\in X \mid \mbox{rank}[\cM(x)]<k\}.$$

The next proposition computes $\Sigma(\cM_D)$ in $\cO_{X\times X}^{2p}$.

\begin{proposition}[\cite{SG}, Proposition 3.7]\label{P4} Let $\cM$ be a sheaf of submodules of $\cOXp$ of generic rank $k$. Then
	\begin{center}$\Sigma(\cM_D)= \Delta(X)\cup(X\times \Sigma(\cM))\cup(\Sigma(\cM)\times X).$\end{center} 
\end{proposition}

We end this section recalling the notion of the analytic projectivization of the Rees algebra of a module. Denote $\cRM$ as the Rees algebra of $\cM$.  Consider the set \begin{center}$\cU(\cM):=\left\{(x,[\ell])\in X\times \bP^{r-1} \mid [\cM(x)]\mbox{ has maximal rank and }\ell\in\mbox{row}[\cM(x)]    \right\}.$\end{center}

\begin{definition}
	The projective analytic spectrum of $\cRM$ is defined as \begin{center}$\Projan(\cRM):=\overline{\cU(\cM)},$\end{center}
	\noindent where the closure is taken on $X\times\bP^{r-1}$.
\end{definition}

The main motivation for this definition is the particular case that $\cM=JM(X)$, the jacobian module of $X$ generated by $\left\{ \frac{\partial F}{\partial z_1},\hdots ,\frac{\partial F}{\partial z_n} \right\}$, where $X$ is defined by an analytic map $F:(\bC^n,0)\rightarrow(\bC^p,0)$. In this setting, suppose $(x,[\ell_1,\hdots ,\ell_n])\in\cU(JM(X))$. So the jacobian matrix of $F$ has maximal rank at $x$, hence $x$ is a smooth point of $X$. Since the rowspace of the jacobian matrix at $x$ is $TX_x^{\perp}$ then $TX_x\subseteq H_{[\ell_1,\hdots ,\ell_n]}$, where $H_{[\ell_1,\hdots ,\ell_n]}$ is the canonical hyperplane induced by $(\ell_1,\hdots ,\ell_n)$. So we can see $\cU(JM(X))$ as a set of pairs $(x,H)$ where $x$ is a smooth point of $X$ and $H$ is a tangent hyperplane. Hence, the fiber of $\Projan(\cRM)$ over the origin captures limiting tangent hyperplanes of $X$.

\section{General results about the Projan of the Rees Algebra of the double}

Let $\cM$ be a sheaf of $\cO_X$-submodules of $\cOXp$ generated by global sections $\{g_1,\hdots ,g_r\}$. Consider the projections $\pi_1,\pi_2:X\times X\rightarrow X$.

For each $x\in X$ we denote $\Projan(\cRM)_x$ the fiber of $x$ with respect to the canonical projection $\pi_{\cM}:\Projan(\cRM)\rightarrow X$.

We denote $\mathbf{2}\cM:=\pi_1^*(\cM)\oplus\pi_2^*(\cM)$ which is an $\cO_{X\times X}$-submodule of $\cO_{X\times X}^{2p}$. So $\Projan(\cRtM)\subseteq X\times X\times\bP^{2r-1}$ and the Proposition \ref{P2.13} tell us that $(\tM)_{(x,x')}=(\cM_D)_{(x,x')}$ provided the point $(x,x')$ is off the diagonal.

We have a simpler description of $\Projan(\cRMD)$ and $\Projan(\cRtM)$. 

\begin{proposition}\label{P2.1}
	Suppose $\cM$ is a sheaf of $\cO_X$-submodules of $\cOXp$. Then \begin{center}$\Projan(\cRMD)=\overline{\left\{ \begin{matrix}
	(x,x',[u+u',(x_1-x_1')u',\hdots ,(x_n-x_n')u'])\mid \\
	x\neq x'\in X, [\cM(x)]\mbox{ and }[\cM(x')]\mbox{ have maximal}\\
	\mbox{rank, }u\in\mbox{row}[\cM(x)]\mbox{ and }u'\in\mbox{row}[\cM(x')]   \\
	\end{matrix}       \right\}}$\end{center}

and 

\begin{center}$\Projan(\cRtM)=\overline{\left\{ \begin{matrix}
	(x,x',[u,u'])\mid [\cM(x)]\mbox{ and }[\cM(x')]\mbox{ have maximal}\\
	\mbox{rank, }u\in\mbox{row}[\cM(x)]\mbox{ and }u'\in\mbox{row}[\cM(x')]   \\
	\end{matrix}       \right\}}.$\end{center}
\end{proposition}

\begin{proof}
	Let us prove the first equality. By Proposition 2.4 of \cite{SG}, we can write a matrix of generators of $\cM_D$ induced by $\{g_1,\hdots ,g_r\}$ on the form
	
	\begin{center}
		$[\cM_D]=\begin{bmatrix}
		w_1  &   0   &  \hdots &  0\\
		\vdots & \vdots &  & \vdots\\
		w_p    &  0  &  \hdots  &  0\\
		w_1'  &  (z_1-z_1')w_1' & \hdots  & (z_n-z_n')w_1'\\
		\vdots  & \vdots  & & \vdots \\
		w_p'  &  (z_1-z_1')w_p' & \hdots  & (z_n-z_n')w_p'\\
		\end{bmatrix}$
	\end{center}
\noindent where $w_1,\hdots ,w_p$ are the rows of $[\cM]$. Let us call $H$ the set under the bar of the equality on the statement. It suffices to prove that $H=\cU(\cM_D)$. Let $(x,x',[u+u',(x_1-x_1')u',\hdots ,(x_n-x_n')u'])\in H$. Since $x\neq x'$ and $[\cM(x)]$ and $[\cM(x')]$ have maximal rank then Proposition 2.7 of \cite{SG} implies that $(x,x')\notin\Sigma(\cM_D)$, so $[\cM_D(x,x')]$ has maximal rank. By hypothesis we can write $u=\sum\limits_{i=1}^p\alpha_iw_i(x)$ and $u'=\sum\limits_{i=1}^p\alpha_i'w_i(x')$, for some $\alpha_i,\alpha_i'\in\bC$. Thus $(x,x',[u+u',(x_1-x_1')u',\hdots ,(x_n-x_n')u'])$ can be written as $$\sum\limits_{i=1}^p\alpha_i(w_i(x),0,\hdots,0)+\sum\limits_{i=1}^p\alpha_i'(w_i(x'),(x_1-x_1')w_i(x'),\hdots,(x_n-x_n')w_i(x'))$$ which is a linear combination of the rows of $[\cM_D(x,x')]$. Hence $H\subseteq\cU(\cM_D)$.

Conversely, let $(x,x',[\ell_1,\hdots ,\ell_r,(\ell_{ij})])\in\cU(\cM_D)$. So $[\cM_D(x,x')]$ has maximal rank then $(x,x')\notin \Sigma(\cM_D)=\Delta(X)\cup(X\times\Sigma(\cM))\cup(\Sigma(\cM)\times X)$. Thus $x\neq x'$, $[\cM(x)]$ and $[\cM(x')]$ have maximal rank. Since $(\ell_1,\hdots ,\ell_r,(\ell_{ij}))$ belongs to the rowspace of $[\cM_D(x,x')]$ then it can be written as $$\sum\limits_{i=1}^p\alpha_i(w_i(x),0,\hdots,0)+\sum\limits_{i=1}^p\alpha_i'(w_i(x'),(x_1-x_1')w_i(x'),\hdots,(x_n-x_n')w_i(x'))$$ \noindent for some $\alpha_i,\alpha_i'\in\bC$. Taking $u:=\sum\limits_{i=1}^p\alpha_iw_i(x)$ and $u':=\sum\limits_{i=1}^p\alpha_i'w_i(x')$ one has $$(\ell_1,\hdots ,\ell_r,(\ell_{ij}))=(u+u',(x_1-x_1')u',\hdots,(x_n-x_n')u')$$
\noindent with $u\in\mbox{row}[\cM(x)]$ and $u'\in\mbox{row}[\cM(x')]$.

The second equality can be proved in an analogous way.
\end{proof}

As a corollary we have an important particular case.

\begin{corollary}
	If $X\subseteq\bC^n$ is an analytic variety defined by $F:(\bC^n,0)\rightarrow(\bC^p,0)$ then 
	\begin{center}$\Projan(JM(X)_D)=\overline{\left\{ \begin{matrix}
		(x,x',[u+u',(x_1-x_1')u',\hdots ,(x_n-x_n')u'])\mid \\
		x\neq x'\mbox{ are smooth points of }X, TX_x\subseteq H_u\\
		\mbox{and }TX_{x'}\subseteq H_{u'}\\
		\end{matrix}       \right\}}.$\end{center}
\end{corollary}

Notice that the fiber of $\Projan(JM(X)_D)$ over the origin captures limiting tangent hyperplanes of $X$ at two different points as these points come together to the origin, taking into account the secant lines defined by these points as they tend to the origin.

\begin{remark}
	Let $x,x'\in X$, $x\neq x'$. Suppose that $u,v\in\mbox{row}[\cM(x)]$ and $u',v'\in\mbox{row}[\cM(x')]$ generate the same element
	$$[u+u',(x_1-x_1')u',\hdots ,(x_n-x_n')u']=[v+v',(x_1-x_1')v',\hdots ,(x_n-x_n')v']$$
	\noindent in the fiber $\Projan(\cRMD)_{(x,x')}$, then $[u,u']=[v,v']$ as elements of $\bP^{2r-1}$. In fact, there exists $\lambda\in\bC-\{0\}$ such that $u+u'=\lambda(v+v')$ and $(x_i-x_i')u'=\lambda(x_i-x_i')v',\forall i\in\{1,\hdots ,n\}$. Since $x_i\neq x_i'$ for some $i$, then $u'=\lambda v'$. Hence, $u=\lambda v$.
\end{remark}

We want to describe $\Projan(\cRMD)_{(x,x')}$ in terms of the sheaf $\tM$. The next lemmas will be useful to do it.

\begin{lemma}\label{L2.4}
	Let $x,x'\in X$, $x\neq x'$. Suppose $x_i\neq x_i'$, for some $i\in\{1,\hdots,n\}$. Let $(x,x',[v,w_1,\hdots ,w_n])\in\Projan(\cRMD)_{(x,x')}$. 
	
	Then:
	
	\begin{enumerate}
		\item [a)] $(x,x',[(x_i-x_i')v-w_i,w_i])\in\Projan(\cRtM)_{(x,x')}$;
		
		\item [b)] If $j\in\{1,\hdots,n\}$ then $w_j=\frac{x_j-x_j'}{x_i-x_i'}w_i$;
		
		\item [c)] If $j\in\{1,\hdots,n \}$ is such that $x_j\neq x_j'$ then $$[(x_j-x_j')v-w_j,w_j]=[(x_i-x_i')v-w_i,w_i].$$
		\end{enumerate}
\end{lemma}

\begin{proof}
	By Proposition \ref{P2.1} we can write $$(x,x',[(x_i-x_i')v-w_i,w_i])=\lim\limits_{t\rightarrow 0} \Phi(t)$$
	\noindent where\\ $\Phi(t)=(\varphi(t),\varphi'(t),[\gamma(t)+\gamma'(t),(\varphi_1(t)-\varphi_1'(t))\gamma'(t),\hdots,(\varphi_n(t)-\varphi_n'(t))\gamma'(t)])$ and for all $t$ close to $0$:
	
	\begin{itemize}
		\item $\varphi(t)\neq\varphi'(t)$;
		
		\item $[\cM(\varphi(t))]$ and $[\cM(\varphi'(t))]$ have maximal rank;
		
		\item $\gamma(t)\in\mbox{row}[\cM(\varphi(t))]$ and $\gamma'(t)\in\mbox{row}[\cM(\varphi'(t))]$.
	\end{itemize}

Further:

\begin{itemize}
	\item $\lim\limits_{t\rightarrow 0}\varphi(t)=x$ and $\lim\limits_{t\rightarrow 0}\varphi'(t)=x'$;
	
	\item $\lim\limits_{t\rightarrow 0}(\gamma(t)+\gamma'(t))=v$;
	
	\item $\lim\limits_{t\rightarrow 0}(\varphi_{\ell}(t)-\varphi'_{\ell}(t))\gamma'(t)=w_{\ell}, \forall \ell\in\{1,\hdots,n\}$. \hspace{2cm}$(\star)$
\end{itemize}

Since $x_i\neq x_i'$ then $\varphi_i(t)-\varphi_i'(t)\neq 0$ for all $t$ sufficiently closed to $0$.
	
	(a) Define $$\beta(t):=(\varphi(t),\varphi'(t),[(\varphi_i(t)-\varphi_i'(t))(\gamma(t)+\gamma'(t))-(\varphi_i(t)-\varphi_i'(t))\gamma'(t),(\varphi_i(t)-\varphi_i'(t))\gamma'(t)]).$$
	
	Clearly $$(x,x',[(x_i-x_i')v-w_i,w_i])=\lim\limits_{t\rightarrow 0}\beta(t).$$
	
	In the other hand, $\beta(t)=(\varphi(t),\varphi'(t),[\gamma(t),\gamma'(t)])\in\cU(\tM)$, for all $t$ close to $0$. Hence, $(x,x',[(x_i-x_i')v-w_i,w_i])\in\Projan(\cRtM)_{(x,x')}$.
	
	(b) It is a consequence of the equation $(\star)$.
	
	(c) It follows immediately from (b).
\end{proof}

\begin{lemma}\label{L2.5}
	Let $x,x'\in X$ and suppose $x\neq x'$. If $(x,x',[u,u'])\in\Projan(\cRtM)_{(x,x')}$ then \begin{center}$(x,x',[u+u',(x_1-x_1')u',\hdots,(x_n-x_n')u'])\in\Projan(\cRMD)_{(x,x')}.$\end{center}
\end{lemma}

\begin{proof}
	By Proposition \ref{P2.1} we can write $$(x,x',[u,u'])=\lim\limits_{t\rightarrow 0}\Phi(t)$$
	\noindent where $\Phi(t)=(\varphi(t),\varphi'(t),[\gamma(t),\gamma'(t)])$ and for all $t$ close to $0$:
	
		\begin{itemize}
		\item $[\cM(\varphi(t))]$ and $[\cM(\varphi'(t))]$ have maximal rank;
		
		\item $\gamma(t)\in\mbox{row}[\cM(\varphi(t))]$ and $\gamma'(t)\in\mbox{row}[\cM(\varphi'(t))]$.
		\end{itemize}
		
		Further: 
		
		\begin{itemize}
			\item $\lim\limits_{t\rightarrow 0}\varphi(t)=x$ and $\lim\limits_{t\rightarrow 0}\varphi'(t)=x'$;
			
			\item $\lim\limits_{t\rightarrow 0}\gamma(t)=u$ and $\lim\limits_{t\rightarrow 0}\gamma'(t)=u'$.
		\end{itemize}
	
	Consider the curve $$\beta(t):=(\gamma(t)+\gamma'(t),[(\varphi_1(t)-\varphi_1'(t))\gamma'(t),\hdots ,(\varphi_n(t)-\varphi_n'(t))\gamma'(t)]).$$
	
	Clearly $$(x,x',[u+u',(x_1-x_1')u',\hdots,(x_n-x_n')u'])=\lim\limits_{t\rightarrow 0}\beta(t)$$
	
	\noindent and Proposition \ref{P2.1} implies that $\beta(t)\in\cU(\cM_D)$, for all $t$ close to $0$.	
\end{proof}

In order to describe the fiber of $\Projan(\cRMD)$ at a point $(x,x')$ off the diagonal in terms of the fiber of $\Projan(\cRtM)$ at this point, and the fibers of $\Projan(\cRM)$ at these two points, we use the \textit{join} of two topological spaces. 

Given $A,B$ topological spaces, we use the classical notation $A\ast B$ to denote the join of these topological spaces, which is, intuitively, the set of lines connecting points of $A$ to $B$. Each element of $A\ast B$ can be written as a formal linear combination $s\mathbf{a}+t\mathbf{b}$, where $\mathbf{a}\in A$, $\mathbf{b}\in B$, and $s,t$ belong to the real closed interval $[0,1]$ and $s+t=1$.

Consider the diagram of canonical inclusions \begin{tikzpicture}
	\matrix (m) [matrix of math nodes,row sep=0.6em,column sep=0.2em,minimum width=0.3em]
	{
	A	&  & B\\
	  & A\ast B   &    \\};
	\path[-stealth]
	(m-1-1) edge node [right] {}  (m-2-2)
	(m-1-3) edge node [right] {}  (m-2-2);
	\end{tikzpicture}. It is well known that this diagram is the homotopy colimit of the diagram of canonical projections \begin{tikzpicture}
	\matrix (m) [matrix of math nodes,row sep=0.2em,column sep=1em,minimum width=0.1em]
	{
		A	& A\times B & B\\};
	\path[-stealth]
	(m-1-2) edge node [right] {}  (m-1-1)
	(m-1-2) edge node [right] {}  (m-1-3);
	\end{tikzpicture}.

\begin{theorem}
	Let $\cM$ be a sheaf of $\cO_X$-submodules of $\cOXp$. Suppose $x,x'\in X$, $x\neq x'$. Then:
	\begin{enumerate}
	
	 \item [a)]$\Projan(\cRMD)_{(x,x')}\cong \Projan(\cRtM)_{(x,x')}$;
	 
	 \item [b)] $\Projan(\cRtM)_{(x,x')}\cong \Projan(\cRM)_x\ast \Projan(\cRM)_{x'}.$
	\end{enumerate}
\end{theorem}

\begin{proof}
	(a )Since $x\neq x'$ then $x_i\neq x_i'$, for some $i$. By Lemma \ref{L2.5}, the map 
	
	\begin{center}
		$\begin{matrix}
			\Gamma: &  \Projan(\cRtM)_{(x,x')} & \longrightarrow & \Projan(\cRMD)_{(x,x')}\\
			 &   (x,x',[u,u']) & \longmapsto &  (x,x',[u+u',(x_1-x_1')u',\hdots,(x_n-x_n')u'])\\
		\end{matrix}$
	\end{center}
\noindent is well defined and Lemma \ref{L2.4} (a) ensures the map 

\begin{center}
	$\begin{matrix}
	\Lambda: & \Projan(\cRMD)_{(x,x')}  & \longrightarrow &\Projan(\cRtM)_{(x,x')} \\
	&   (x,x',[v,w_1,\hdots,w_n]) & \longmapsto &  (x,x',[(x_i-x_i')v-w_i,w_i])\\
	\end{matrix}$
\end{center}

\noindent is well defined (and it does not depend on $i$ such that $x_i\neq x_i'$). It is easy to verify that $\Lambda\circ\Gamma=\mbox{id}$ and one can use Lemma \ref{L2.4}(b) to conclude that $\Gamma\circ\Lambda=\mbox{id}$.

(b) Consider the inclusions

\begin{center}
	$\begin{matrix}
		i_x: &  \Projan(\cRM)_x &  \rightarrow & \Projan(\cRtM)_{(x,x')}\\
		      &   (x,[u])   & \mapsto  &  (x,x',[u,0])
	\end{matrix}$
\end{center} and \begin{center}
$\begin{matrix}
i_{x'}: &  \Projan(\cRM)_{x'} &  \rightarrow & \Projan(\cRtM)_{(x,x')}\\
&   (x',[u'])   & \mapsto  &  (x,x',[0,u'])
\end{matrix}$
\end{center}

Then the diagram 
\begin{center}
\begin{tikzpicture}
\matrix (m) [matrix of math nodes,row sep=2em,column sep=0.1em,minimum width=0.1em]
{
	\Projan(\cRM)_x	&   & \Projan(\cRM)_{x'}\\
	 & \Projan(\cRtM)_{(x,x')}  &  \\};
\path[-stealth]
(m-1-1) edge node [below] {$i_x$}  (m-2-2)
(m-1-3) edge node [below] {$i_{x'}$}  (m-2-2);
\end{tikzpicture}
\end{center}\noindent is a homotopy colimit of the diagram of canonical projections.

\noindent \begin{tikzpicture}
\matrix (m) [matrix of math nodes,row sep=0.2em,column sep=1em,minimum width=0.1em]
{
	\Projan(\cRM)_{x}	& \Projan(\cRM)_{x}\times \Projan(\cRM)_{x'} & \Projan(\cRM)_{x'}\\};
\path[-stealth]
(m-1-2) edge node [right] {}  (m-1-1)
(m-1-2) edge node [right] {}  (m-1-3);
\end{tikzpicture}

Therefore, $\Projan(\cRtM)_{(x,x')}\cong \Projan(\cRM)_x\ast \Projan(\cRM)_{x'}.$
\end{proof}

Suppose $\cM$ is a sheaf of $\cO_X$-submodules of $\cOXp$ generated by global sections $g=\{g_1,\hdots,g_r\}$. We define the \textit{reduced double} of $\cM$ with respect to $g$ as the sheaf of $\cO_{X\times X}$-submodules $\cM_{D-}$ of $\cO_{X\times X}^{2p}$ generated by $\{(g_1)_D,\hdots,(g_r)_D\}$. 

In this way, we define $\cM_{\widehat{D}}:=\cM_{D-}+(0\oplus \pi_2^*(\cM))\subseteq\cO_{X\times X}^{2p}$, which is called the \textit{extended double of} $\cM$ with respect to $g$. Clearly, $$\cM_{D-}\subseteq\cM_D\subseteq \cM_{\widehat{D}}\subseteq \tM.$$

\begin{proposition}
	$\Projan(\cR(\tM))$ is naturally isomorphic to $\Projan(\cR(\cM_{\widehat{D}}))$, i.e, there exists an isomorphism $\Projan(\cR(\tM))\rightarrow\Projan(\cR(\cM_{\widehat{D}}))$ such that the following diagram is commutative: 
	
	\begin{center}
	 \begin{tikzpicture}
	\matrix (m) [matrix of math nodes,row sep=2.5em,column sep=1em,minimum width=0.1em]
	{
		\Projan(\cR(\tM))	&  & \Projan(\cR(\cM_{\widehat{D}}))\\
		X\times X  &      &  X\times X  \\};
	\path[-stealth]
	(m-1-1) edge node [left] {$\pi_{\tM}$}  (m-2-1)
	(m-1-1) edge node [above] {$\sim$}  (m-1-3)
	(m-2-1) edge node [below] {$\mbox{id}$}  (m-2-3)
	(m-1-3) edge node [right] {$\pi_{\cM_{\widehat{D}}}$}  (m-2-3);
	\end{tikzpicture}
	\end{center}
\end{proposition}

\begin{proof}
	Let $w_1,\hdots,w_p$ be the rows of $[\cM]$. Then we can write \begin{center} 	
	$[\cM_{\widehat{D}}]=\left[\begin{matrix}
	 w_1 & \mid &  0   \\
	 \vdots  & \mid     &   \vdots \\
	  w_p & \mid &  0 \\
	   w_1' & \mid & w_1'\\ 
	   \vdots  & \mid     &   \vdots \\
	    w_p' & \mid &  w_p'\\
	\end{matrix}\right]$ and $[\tM]=\left[\begin{matrix}
	w_1 & \mid &  0   \\
	\vdots  & \mid     &   \vdots \\
	w_p & \mid &  0 \\
	0 & \mid &  w_1'\\ 
	\vdots  & \mid     &   \vdots \\
	0 & \mid &  w_p'\\
	\end{matrix}\right]$.\end{center}
	
	Notice that if $(x,x',[u,u'])\in\cU(\cM_{\widehat{D}})$ then $[\cM_{\widehat{D}}(x,x')]$ has maximal rank and $(u,u')\in\mbox{row}[\cM_{\widehat{D}}(x,x')]$. Thus, $[\cM(x)]$ and $[\cM(x')]$ have maximal rank. It is easy to conclude that $u-u'\in\mbox{row}[\cM(x)]$ and $u'\in\mbox{row}[\cM(x')]$. Hence, taking limits, the map 
	
	\begin{center}
	$	\begin{matrix}
		\Lambda:  &  \Projan(\cR(\cM_{\widehat{D}})) & \rightarrow  & \Projan(\cRtM)\\	
		   &   (x,x',[u,u']) & \mapsto  & (x,x',[u-u',u'])
		\end{matrix}$ 
	\end{center}\noindent is well defined. Analogously, the map

	\begin{center}
	$	\begin{matrix}
	\Gamma: &\Projan(\cRtM) &    \rightarrow  &\Projan(\cR(\cM_{\widehat{D}}))\\	
	&   (x,x',[u,u']) & \mapsto  & (x,x',[u+u',u'])
	\end{matrix}$ 
\end{center}\noindent is well defined. Hence $\Lambda\circ\Gamma=\mbox{id}$, $\Gamma\circ\Lambda=\mbox{id}$ and $\pi_{\cM_{\widehat{D}}}\circ \Gamma=\pi_{\tM}$.
\end{proof}

If we consider the sheaf of ideals $\cF$ on $\Projan(\cR(\cM_{\widehat{D}}))$ generated by $$\{T_1,\hdots ,T_r\}\cup\{(z_i-z_i')T_j'\mid i\in\{1,\hdots,n\}\mbox{ and }j\in\{1,\hdots,r \}\},$$\noindent  where $T_1,\hdots,T_r,T_1',\hdots,T_r'$ are the homogeneous coordinates on $\bP^{2r-1}$, then we can consider the blowup $B_{\cF}(\Projan(\cR(\cM_{\widehat{D}})))$. As one can see in \cite{KT}, the inclusion of graded $\cO_{X\times X}$-algebras $\cR(\cM_D)\hookrightarrow\cR(\cM_{\widehat{D}})$ induces a commutative diagram

\begin{center}
	 \begin{tikzpicture}
	\matrix (m) [matrix of math nodes,row sep=4em,column sep=3.5em,minimum width=0.5em]
	{
	B_{\cF}(\Projan(\cR(\cM_{\widehat{D}})))	& \Projan\cR(\cM_{\widehat{D}}) & \Projan\cR(\tM)\\
		\Projan\cRMD &      X\times X  &  \\};
	\path[-stealth]
	(m-1-1) edge node [above] {$\pi_{\cF}$}  (m-1-2)
	(m-1-2) edge node [above] {$\sim$}  (m-1-3)
	(m-1-1) edge node [left] {$\mathbf{p}$}  (m-2-1)
	(m-2-1) edge node [below] {$\pi_{\cM_D}$}  (m-2-2)
	(m-1-3) edge node [right] {$\pi_{\tM}$}  (m-2-2)
	(m-1-2) edge [dashed] node [right] {}  (m-2-1)
	(m-1-2) edge node [right] {$\pi_{\cM_{\widehat{D}}}$}  (m-2-2);
	\end{tikzpicture}
\end{center}

\begin{remark}\normalfont 
	Notice that $\pi_{\cM_{\widehat{D}}}(V(\cF))\subseteq \Delta(X)$. In fact, if $(x,x',[u,u'])\in V(\cF)$, then $T_j(u)=0$ and $(x_i-x_i')T_j'(u')=0$, $\forall i,j$. So $u=0$ and necessarily $u'\neq 0$, thus $T_j'(u')\neq 0$, for some $j$. Hence, $x_i-x_i'=0$, for every $i$ and $x=x'$. In particular, $\mathbf{p}$ takes the exceptional divisor of $B_{\cF}(\Projan(\cR(\cM_{\widehat{D}})))$ in the exceptional fiber of $\Projan\cRMD $.
	
	An element of $V(T_1,\hdots,T_r)\subseteq\Projan\cR(\cM_{\widehat{D}})$ can be written on the form $(x,x',[0,u'])$. If we assume that $[\cM(x)]$ and $[\cM(x')]$ have maximal rank and $(0,u')\in\mbox{row}[\cM_{\widehat{D}}(x,x')]$ then $u'\in\mbox{row}[\cM(x)]\cap\mbox{row}[\cM(x')]$. In the case where $\cM=JM(X)$, the element $(x,x',[0,u'])$ satisfies $TX_x\subseteq H_{u'}$ and $TX_{x'}\subseteq H_{u'}$, i.e, $H_{u'}$ is a common tangent hyperplane to $x$ and $x'$.
\end{remark}

\section{The fiber of the Projan over the origin for curves in hypersurfaces}

Our approach is to work on irreducible curves $C$ in $X$ and calculate the fiber over the origin in $C\times C$ of the restriction of $\mbox{Projan}(\cR((JM(X))_D))$  to $C\times C$.

The computation is completed for $X$ an irreducible plane curve. Some partial results are obtained.

Given an irreducible curve $C$ on $X$, it has a normalization $\eta:\bC\rightarrow C$. If $\Phi:(\bC,0)\rightarrow (C,0)$ is any map then $\Phi=\eta\circ\varphi$ for some $\varphi:(\bC,0)\rightarrow (\bC,0)$. If $C$ is irreducible then $\eta$ is $1-1$ and an embedding except at $0$. So the fiber of $\mbox{Projan}(\cR((JM(X))_D))$ over $(0,0)$on $C\times C$ is isomorphic to the fiber over $(0,0)$ of the pullback to $\bC\times\bC$.

\begin{lemma}
	Suppose $\Phi:(\bC,0)\rightarrow C\times C$ and $\Phi=(\Phi_1,\Phi_2)$. Then $\Phi=(\eta\circ\varphi_1,\eta\circ\varphi_2)$ for some $\varphi_1,\varphi_2:(\bC,0)\rightarrow(\bC,0)$.
\end{lemma}

\begin{proof}
	It is obvious once each $\Phi_i$ factors through $\eta$.
\end{proof}

Now $JM(X)_D$ has $2$ natural submodules: $JM(X)_{D-}$ which is the submodule generated by the doubles of the partial derivatives of $F$, where $X=f^{-1}(0)$, and $JI_{\Delta}:=0\oplus I_{\Delta}JM(X)$.

The fiber of $\mbox{Projan}(\cR(JI_{\Delta}))$ over the origin consists of pairs $(H,\ell)$ where $H$ is a limiting tangent hyperplane and $\ell$ is a limiting secant line realized using a curve on which $H$ is the limiting tangent hyperplane and any other curve. The fiber of $\mbox{Projan}(\cR(JM(X)_{D-}))$ we describe in the case where $X$ is a hypersurface. We need the notion of the second intrinsec derivative of $X$ to do this. 
                                                                                                     
Suppose $X=f^{-1}(0)$, $x\in X$ a smooth point of $X$; then the second intrinsec derivative of $f$ is the map $$\hat{D}_2f_x:k_x\rightarrow C_x$$ where $k_x:=\ker(Df_x)$, $C_x:=\mbox{coker}(Df_x)$ and $\hat{D}_2f_x$ defined by restricting $D_2f$ to $k_x$ and projecting to $C_x$. We can also view $\hat{D}_2f_x$ as a map from $k_x$ to $\bC^n$ but then it is only defined up to multiples of $\nabla f(x)$.

\begin{theorem}\label{T1}
	Suppose $C\subset X$ is an irreducible curve in a hypersurface $X$. Then the fiber of $\mbox{Projan}(\cR((J(X)_{D-})\mid_{ C\times C}))$ over the origin contains $\bP(Span(H,L))$, where $H$ is limiting conormal of $X$ along $C$ and $L$ is the limit of $\hat{D}_2f_x\mid_{TC_x}$. 
\end{theorem}

\begin{proof}
	We calculate limits of the form $$\lim\limits_{t\rightarrow 0} \frac{1}{t^{\ell}}\left(\psi\cdot \left[\begin{matrix}
	Df\circ\eta\circ\varphi_1\\
	Df\circ\eta\circ\varphi_2
	\end{matrix}\right]\right)$$
	
\noindent where the order in $t$ of $\psi\cdot \left[\begin{matrix}
Df\circ\eta\circ\varphi_1\\
Df\circ\eta\circ\varphi_2
\end{matrix}\right]$ is $\ell$. Notice that $Df\circ\eta(t)$ has the form $$\tau t^k+\mu t^{k+r}+\hdots  \eqno{(\star)}$$

\noindent where $\tau \in\bC^n$ is the limiting conormal vector. We can choose $\varphi$ so that $Df\circ\eta\circ\varphi(t)$ has the form $(\star)$ with $\mu\nparallel \tau$, using $\eta\circ\varphi$ instead of $\eta$ if necessary. Let $\varphi_1(t)=t$ and $\varphi_2(t)=ct+\hdots $, $c$ not an $r$-th root of $1$. Now, $\lim\limits_{t\rightarrow 0} \frac{1}{t^{k}}\left((1,0)\cdot \left[\begin{matrix}
Df\circ\eta\circ\varphi_1\\
Df\circ\eta\circ\varphi_2
\end{matrix}\right]\right)=\tau$ and $$\lim\limits_{t\rightarrow 0} \frac{1}{t^{k+r}}\left(-\frac{\varphi_2^k(t)}{t^k}(Df\circ\eta\circ\varphi_1)+Df\circ\eta\circ\varphi_2 \right)$$$$=\lim\limits_{t\rightarrow 0} \frac{1}{t^{k+r}}\left(-\frac{\varphi_2^k}{t^k}\tau t^k-\frac{\varphi_2^k}{t^k}\mu t^{k+r}+\varphi_2^k\tau +\varphi_2^{k+r}\mu+\cdots\right)=\lim\limits_{t\rightarrow 0}\frac{(\varphi_2^{k+r}-\varphi_2^kt^r)}{t^{k+r}}\mu$$$$=\lim\limits_{t\rightarrow 0}c^k(c^r-1)\mu\parallel \mu.$$

Now we describe $\mu$. If we work at $\eta(t)$, we have $$\star\star:=\lim\limits_{h\rightarrow 0}\frac{Df\circ\eta(t+h)-Df\circ\eta(t)}{h}=D_2f(\eta(t))\cdot\eta'(t).$$

Since $C\subset X$, $\eta'\in\ker Df(\eta(t))$, hence this is a value of the second intrinsic derivative of $\eta(t)$. Now we calculate the limit using the normal form $$\star\star=\lim\limits_{h\rightarrow 0}(t+h)^k\tau-t^k\tau+(t+h)^{k+r}\mu-t^{k+r}\mu+\cdots=k\tau t^{k-1}+(k+r)t^{k+r-1}\mu+\cdots$$

Since $t\neq 0$, substract $\frac{k}{t}Df(\eta(t))$ which is a multiple of $Df(\eta(t))$. This gives $((k+r)-k)t^{k+r-1}\mu+\cdots$. Now we projectivize and take the limit as $t\rightarrow 0$, and get $r\mu$.
\end{proof}

\begin{corollary}
	If $X$ is a plane curve so that $C=X$ then $\bP(\mbox{span}(\tau,\mu))$ is the fiber of $\textrm{Projan}(\cR(J(f)_{D-}))$ over $(0,0)$.
\end{corollary}

\begin{proof}
	The fiber over $(0,0)$ in this case lies in $(0,0)\times\bP^1$.
\end{proof}

\begin{remark}
	Note that if $C\subseteq\bC^n$, $n>2$ and $\nu$ independent of $\tau$ and $\mu$ then the fiber is $\bP(\mbox{span}(\tau,\mu,\nu))$, provided $r\neq 1$, which has dimension $2$, the maximum possible. $\nu$ is the limit of the $2^{\mbox{\tiny{nd}}}$ intrinsec derivative of $Df$ (i.e the intrinsec derivative of the intrinsec derivative).
\end{remark}

\begin{example}
	Consider $f(x,y,z)=x^2+y^3+z^4$ and $\varphi(t)=(i\sqrt{8}t^6,-2t^4,2t^3)$. Then, $Df\circ\varphi(t)=(2\sqrt{8}i,0,0)t^6+(0,12,0)t^8+(0,0,32)t^9$ \noindent and \begin{center}
		$\tau=(1,0,0)$, $\mu=(0,1,0)$, $\nu=(0,0,1)$.
	\end{center}
\end{example}

There is an important case where we can describe the fiber of $\Projan(\cR(J(f)_D\mid_{C\times C}))$. If $X\subseteq\bC^n,0$ is a hypersurface, $X=f^{-1}(0)$ then the Gauss map $G$ of $X$ \begin{center}
	$G:X\setminus\mbox{Sing}(X)\rightarrow\bP^{n-1}$
\end{center}
\noindent is defined by $G(x)=\langle Df(x)\rangle$.

\begin{theorem}\label{T2}
	Suppose $f$ defines a hypersurface $X$, $C$ an irreducible curve on $X$ with multiplicity 2 on the origin, which is not a line, $G\mid_{C-\{0\}}$ extends to an embedding on $C$. Then the fiber over $(0,0)$ of $\Projan(\cR(J(f)_D\mid_{C\times C}))$ is $\Projan(\cR(J(f)_{D-}))_{(0,0)}$.  
\end{theorem}

\begin{proof}
	As before we can assume $\Phi=(\eta\circ\varphi_1,\eta\circ\varphi_2)$, $\eta$ the normalization of $C$.
	
	\underline{Base case}: Assume $\varphi_1=\mbox{id}$ and $\varphi_2=ct+\cdots$, $c\neq 1$.
	
	We may suppose $\deg$$\frac{\partial f}{\partial z_1}$ along $C$ is minimal among $\left\{\deg\frac{\partial f}{\partial z_i}\right\}$.
	
	So \begin{center}
		$\left[\begin{matrix}
		Df(\eta(t))\\
		Df(\eta(\varphi_2(t))) 
		\end{matrix}\right]=\left[\begin{matrix}
		\frac{\partial f}{\partial z_1}(\eta(t))\left((1,D)+(0,tE)+\cdots \right) \\
		\frac{\partial f}{\partial z_1}(\eta(\varphi'_2(t)))\left((1,D)+(0,\varphi_2(t)E)+\cdots\right)
		\end{matrix}\right]$.
	\end{center}

	Notice that $G(\eta(t))=\langle (1,D)+(0,tE)+\cdots\rangle$. The condition that the Gauss map is a local embedding means that $E\neq 0$. Now $\xi=\frac{f_{z_1}\circ\eta(\varphi_2(t))}{f_{z_1}\circ\eta(t)}$ is a unit, apply $(-\xi,1)$ to the matrix of generators of $J(f)_D$: we get \begin{center}
		$f_{z_1}\circ\eta\circ\varphi_2\left[0,((c-1)t+\cdots)E+\cdots, \hdots ,\eta_j(t)-\eta_j(\varphi_2(t)),\hdots\right]$.
	\end{center}

	Since $C$ is not a line the terms $\eta_j(t)-\eta_j(\varphi_2(t))$ must have degree at least $2$ so taking the limit as $t\rightarrow 0$ and projectivizing we get $\langle (0,E,0)\rangle$. Thus in the base case the limits are elements of $\bP(\mbox{span}\{(1,D,0),(0,E,0)\})$.
	
	In the general case the second row tends to $(1,D,0)$ no matter what $\varphi_2$ is, so there needs the cancellation using the first row; but as in the Base case, the order of the coefficients of the $E$ term is less than the degree of the terms $\eta_i\circ\varphi_1-\eta_i\circ\varphi_2$.
\end{proof}

\begin{remark}
	The case where the Gauss map is an embedding is equivalent to the number $r$ in Theorem \ref{T1} being $1$, in which case $c=1$ is the only root of unity. This explains why the fiber of $\Projan(\cR(J(f)_{D-}))$ only has dimeinsion $1$ instead of $2$.  
\end{remark}

Now we specialize to the case where $C$ is an irreducible plane curve. We assume $\eta(t)=(t^n,t^{B_1}+\cdots+a_{B_2}t^{B_2}+\cdots)$, $B_1$ not a multiple of $n$, $B_2$ not in the ideal generated by $\{n,B_1\}$, terms of $\eta_2$ between $t^{B_1}$ and $a_{B_2}t^{B_2}$ have degree which are multiples of $B_1$. We call this the \textit{standard normalization}.

For this part, it is helpful to have Teissier's notes in complex curves singularities \cite{Teissier}, p. 19-20 handy. We refine our previous normal form in a useful way.

\begin{proposition}\label{P1}
	Given an irreducible plane curve $C$ defined by $f(x,y)=0$, if $\eta$ is the standard normalization and $\Phi=(\eta,\eta\circ\varphi)$ then $\Phi^*(J(C)_D)$ has a matrix of generators of form
	
	\begin{center}
		$\left[\begin{matrix}\vspace{0.5cm}
		-f_y\circ\eta\frac{\frac{d\eta_2}{dt}}{\frac{d\eta_1}{dt}} & f_y\circ\eta &   0\\
		-f_y\circ\eta(\varphi(t))(\frac{\frac{d\eta_2}{dt}}{\frac{d\eta_1}{dt}}\circ\varphi)  & f_y\circ\eta(\varphi(t)) &  f_y\circ\eta(\varphi(t))(\eta_1(t)-\eta_1(\varphi(t)),\eta_2(t)-\eta_2(\varphi(t)))\\		
		\end{matrix}\right]$
	\end{center}
\end{proposition}

\begin{proof}
	We have that $f(\eta(t))\equiv 0$ so $Df(\eta(t))\cdot\eta'(t)=0$. Notice that $\frac{\frac{d\eta_2}{dt}}{\frac{d\eta_1}{dt}}$ is analytic, then $f_x\circ\eta(t)=-f_y\circ\eta(t)\frac{\frac{d\eta_2}{dt}}{\frac{d\eta_1}{dt}}$. Now the result follows from Proposition 2.3 of \cite{SG}.
\end{proof}

The lowest degree term in the above matrix is $f_y\circ\eta(t)$. Shortly we will use the row $1$ to cancel the corresponding term in row $2$. Projectivizing we need to understand

\begin{center}
	$\langle \frac{\frac{d\eta_2}{dt}}{\frac{d\eta_1}{dt}}-\frac{\frac{d\eta_2}{dt}}{\frac{d\eta_1}{dt}}\circ\varphi(t),0,(\eta_1(t)-\eta_1(\varphi(t)),\eta_2(t)-\eta_2(\varphi(t))) \rangle$.
\end{center}

The first term measures the difference in slopes of the tangent lines at $\eta(t)$ and $\eta(\varphi(t))$ while the remaining terms are the coordinates of the corresponding secant line.

Since the dimension of $C\times C$ is $2$ and the generic rank of $J(f)_D$ is $2$, the dimension of $\Projan(\cR(J(f)_D))$ is $3$; we expect the fiber over $(0,0)$ to have dimension $2$. The fiber over $(0,0)$ lies in $\bP^3$ by our normal form. In this setting we call the first two coordinates of $\bP^3$ conormal coordinates, the third one the tangent coordinate and the fourth one the normal coordinate. The set of points in $\bP^3$ where the normal components is zero, we call the \textit{tangent component}. It is the smallest convex set in $\bP^3$ which contains the fiber of $\Projan(\cR(J(f)_{D-}))$ and the tangents line to $C$ which is just $(0,0,1,0)$ for our normal form.

\begin{theorem}\label{T3}
	Suppose $C$ is an irreducible plane curve of multiplicity greater than $1$. Let $\eta$ be the standard normalization for $C$, with the associated generators of $J(f)_D$. Then the fiber of $\Projan(\cR(J(f)_{D}))$ over $(0,0)$ is:
	
	\begin{enumerate}
		\item [i)] $\bP^1\times(0,0)$, if $B_1=n+1$ (which is the fiber of $\Projan(\cR(J(f)_{D-}))$);
		
		\item [ii)] The tangent component, if $B_1>n+1$.
	\end{enumerate}
\end{theorem}    

\begin{proof}
	(i) If $B_1=n+1$, then $f=ax^{n+1}-y^n$ and the degrees of $f_x\circ\eta,f_y\circ\eta$ differ by $1$, hence the Gauss map $\langle f_x\circ\eta,f_y\circ\eta$ is a local embedding, and we apply Theorem \ref{T2}.
	
	(ii) We first show that the tangent component is in the fiber. Denote the matrix of Proposition \ref{P1} by $SN_D$. There are two cases:
	
	\underline{Case A}: Suppose $B_1>2n$.
	
	Consider 
		$$(1,0)\cdot SN_D=f_y\circ\eta(t)\left(-\frac{\frac{d\eta_2}{dt}}{\frac{d\eta_1}{dt}},1,0,0\right) \eqno\mathbf{(1)}$$
		
		The projective limit of $\mathbf{(1)}$ is $(0,1,0,0)$. Consider the equation  
		
		$$ \left(-\frac{f_y\circ\eta\circ\varphi(t)}{f_y\circ\eta(t)},1\right)\cdot SN_D\eqno\mathbf{(2)}$$
		
		$$=-f_y\circ\eta\circ\varphi(t)\left(\frac{\frac{d\eta_2}{dt}}{\frac{d\eta_1}{dt}}(t)-\frac{\frac{d\eta_2}{dt}}{\frac{d\eta_1}{dt}}(\varphi(t)),0,\eta_1(t)-\eta_1(\varphi(t)),\eta_2(t)-\eta_2(\varphi(t))\right).$$
		
		Now \begin{center}
			$\frac{\frac{d\eta_2}{dt}}{\frac{d\eta_1}{dt}}=\frac{a_{B_1}}{n}t^{B_1-n}+\cdots+\frac{a_{B_2}}{n}t^{B_2-n}+\cdots$.
		\end{center}
	
	So $\mathbf{(2)}$ becomes 
	
	$-f_y\circ\eta\circ\varphi(t)(\frac{a_{B_1}}{n}t^{B_1-n}+\cdots +\frac{a_{B_2}}{n}t^{B_1-n}+\cdots -(\frac{a_{B_1}}{n}{\varphi}^{B_1-n}+\cdots +\frac{a_{B_2}}{n}{\varphi}^{B_2-n}+\cdots ),  0,t^n-\varphi^n,(t^{B_1}-\varphi^{B_1})+\cdots +(a_{B_2}t^{B_2}-a_{B_2}\varphi^{B_2})+\cdots )$.
	
	Since $B_1>2n$, so $B_1-n>n$, let $\phi(t)=ct+dt^s$ where $c$ is an $n^{\mbox{\tiny{th}}}$ root of unity, but not a $B_1$-st root of unity, hence not a $(B_1-n)$-th root of unity, $d$ arbitrary. The degrees of the leading terms which are non-zero are $(B_1-n,n+s-1,B_1)$. If we choose $s$ to solve the equation $B_1-n=n+s-1$, i.e $s=B_1-2n+1$ we get that the projective limit of $(At^{B_1-n}\mathbf{(1)}+B\mathbf{(2)})$ is $\langle B\frac{B_1}{n}(c^{B_1-n}-1),A,-ndB,0   \rangle$. This show that a tangential component is in the fiber since the component is a closed set.
	
	\underline{Case B}: Suppose $2n>B_1>n$.
	
	This is similar to case (A). Let $\phi(t)=ct+dt^s$ where $c$ is a $(B_1-n)$-th root of unity, not an $n$-th root of unity, $s$ solves $B_1-n+s-1=n$, i.e $s=2n+1-B_1$.
	
	Then the degrees of the leading terms of the non-zero elements of $\mathbf(2)$ are $(B_1-n+s-1,n,B_1)$, so the projective limit of $\langle At^n\mathbf{(1)}+B\mathbf{(2)}  \rangle$ is $\langle Bd\frac{B_1}{n},A,B(1-d^n),0 \rangle$.
	
	Now we show in both cases (A) and (B) that in general in the limit the normal coordinate is zero. We may assume after a coordinate change in $t$ that $$\varphi_1=t^s, \varphi_2=c_{s_1}t^{s_1}+\cdots $$
	
	Then
	 $$\eta_1\circ\varphi_1-\eta_1\circ\varphi_2=t^{ns}-c_{s_1}^nt^{s_1n}+\cdots $$
	 $$\eta_2\circ\varphi_1-\eta_2\circ\varphi_2=a_{B_1}t^{B_1s}-a_{B_1}c_{s_1}^{B_1s}t^{B_1s}+\cdots $$
	 
	 In order for the normal coordinate to be non-zero in the limit $s_1=s$ and $c_s=c_{s_1}$ must be an $n$-th root of unity, so assume $\varphi_2=c_st^s+c_rt^r+\cdots $, $c_s$ a root of unity of order $n,B_1,\hdots,B_{j-1}$, not $B_j$. Consider
	 $$\left(\frac{\frac{d\eta_2}{dt}}{\frac{d\eta_1}{dt}}\circ t^s -\frac{\frac{d\eta_2}{dt}}{\frac{d\eta_1}{dt}}\circ\varphi_2,0,\eta_1\circ t^s-\eta_1\circ\varphi_2,\eta_2\circ t^s-\eta_2\circ\varphi_2\right).$$
	 
	 Then the degrees of the leading non-zero terms are $$\left(\min\limits_{j}\{(B_1-n-1)s+r,(B_j-n)s\},(n-1)s+r,\min\limits_{j}\{(B_1-1)s+r,B_js\}\right).$$
	 
	 Note that $(B_1-1)s+r>(n-1)s+r$ and $s(B_j-n)<B_js$, $\forall s$. This shows that if $c_s$ is a root of unity $\forall B_j$, i.e, $c_s=1$. 
	 
	 Hence in the limit the normal coordinate is zero. Then the only way for a limit with non-zero normal coordinate is to assume $r$ large enough so that $(n-1)s+r>B_js$ and use $d$ to satisfy the appropriate equation so that $(B_j-n)s$ is cancelled in the first coordinate. This implies that $r$ satisfies $(B_1-n-1)s+r=(B_j-n)s$, i.e, $r=(B_j-B_1+1)s$. But then the leading term of the tangent coordinate is $(n-1)s+(B_j-B_1+1)s=(B_j+(n-B_1))s<(B_j)s$. So in the limit the normal coordinate is zero.
\end{proof}

Now we describe the fiber of $\projmxd$ over $(x,x)$, where $x$ is a smooth point of $X$. For simplicity, assume that $X\subseteq\bC^n$ is a hypersurface defined by $F:(\bC^n,0)\rightarrow(\bC,0)$ and $x=0$. For $v\in TX_0$, $v\neq 0$, let $v\cdot DF(0)$ denote $\left(v_i\frac{\partial F}{\partial z_j}\right)\in\bC^{n^2}$. Note that the elements of $\projmxd$ lie in $\bP(\bC^n\oplus\bC^{n^2})$ and $\bC^n$ and $\bC^{n^2}$ can be canonically embedded in $\bC^n\oplus\bC^{n^2}$, so we can regard $DF(0)$ and $v\cdot DF(0)$ as vectors in $\bC^n\oplus\bC^{n^2}$.

Let $\Gamma(\hat{D}_2F)$ denote the vector space spanned by $$\{\hat{D}_2F(0)v\oplus v\cdot DF(0)\mid v\in TX_0\}.$$

\begin{proposition}
If $0$ is a smooth point of $X$ then the fiber of $\projmxd$ over $(0,0) $ is the $\bP(\mbox{span}(DF(0),),\Gamma(\hat{D}_2F))$.
\end{proposition}

\begin{proof}
	Since $X$ is smooth at $0$, by choosing an appropriate linear projection we can view it as the graph of a function $f$, so that $X$ is defined by $z_n-f(z_1,\hdots,z_{n-1})=F$. Note that the second intrinsec derivative of $F$ at $0$ applied to $v\in TX_0$ is $-Df(0)v$.
	
	If $\phi(t)$ is a curve on $X$, $\phi$ has the form $(\varphi_1,\varphi_2,f(\varphi_1,\varphi_2))$, so the leading term of $\phi$ is the leading term of $(\varphi_1,\varphi_2)$. The leading term of $\phi\cdot DF$ is $v\cdot t^k\cdot\frac{\partial F}{\partial z}$, where $v\cdot t^k$ is the leading term of $(\varphi_1,\varphi_2)$. Let $\phi_1(t)=0$ and $\phi_2(t)=\phi$.  The leading term of $$(DF\circ \phi(t)-DF(0),\phi(t)\cdot DF(\phi(t)))$$\noindent is $(-D_2f(0)(vt^k),t^kv\cdot DF(0))$. So the leading term of $$(-1+at^k,1)\cdot\begin{bmatrix}
		DF(0) & 0\\
		DF\circ\phi & (\phi-0)\cdot DF(\phi)
	\end{bmatrix}$$\noindent is $ta^k(DF(0),0)+(-D_2f(0)(vt^k),t^kv\cdot DF(0))$. This shows that the fiber contains the desired space. 

If we replace $\phi_1$ by any other curve, then suppose that the multiplicity of $\phi_2-\phi_1$ is $k$; then again the leading term of $$(DF\circ(\phi_2)-DF\circ(\phi_1),(\phi_2-\phi_1)\cdot DF\circ\phi_2)$$\noindent is $(-D_2f(0)(t^kv),t^kv\cdot DF(0))$ and we get the same result as before (note that $\hat{D}_2F$ is $-D_2f(0))$.
\end{proof}

\section*{Acknowledgements}

Thiago da Silva is funded by CAPES grant number 88887.909401/2023-00 and CAPES grant number 88887.897201/2023-00.

\section*{Data Availability Statement}

This article has no associated data.

\section*{Conflict of interest}

This article has no conflict of interest.

\end{document}